\documentclass[12pt]{amsart}
\usepackage[margin=1in]{geometry}
\usepackage{newtxtext}
\usepackage{newtxmath}
\usepackage[alphabetic]{amsrefs}
\usepackage{graphicx}
\usepackage{caption}
\usepackage{tikz}
	\usetikzlibrary{decorations.pathreplacing}
	\usetikzlibrary{decorations.pathmorphing}
	\usetikzlibrary{patterns}
\usepackage{setspace}
\usepackage[hidelinks]{hyperref}

\newcommand{\cd}{\cdot}
\newcommand{\ra}{\rightarrow}
\newcommand{\pr}{\prime}
\newcommand{\de}{\partial}

\newcommand{\abs}[1]{\left\lvert #1 \right\rvert}

\newtheorem{theorem}{Theorem}

\theoremstyle{definition}

\theoremstyle{remark}

\begin{document}

\title{How Many Links Fit in a Box?}
\date{\today\vspace{0.5\baselineskip}}

\author{Michael H. Freedman}
\address{\hspace{-\parindent}Michael H. Freedman}
\email{freedmanm@google.com, mfreedman@cmsa.fas.harvard.edu}

\begin{abstract}
	In an earlier note \cite{freedman24} it was shown that there is an upper bound to the number of disjoint Hopf links (and certain related links) that can be embedded in the unit cube where there is a fixed separation required between the components within each copy of the Hopf link. The arguments relied on multi-linear properties of linking number and certain other link invariants. Here we produce a very similar upper bound for all non-trivial links by a more-general, entirely geometric, argument (but one which, unlike the original, has no analog in higher dimensions). Shortly after the initial paper, \cite{mp24} proved lower bounds which still provide a converse to our Theorem 1 in the case that only a bounded number of link types appear among the set $\{L_i\}$ as $N$ increases.
\end{abstract}

\maketitle

\section{Embedding Links}

We are concerned with smooth links embedded inside the unit cube $I^3$. For the purposes of this note, a \emph{true link} is a link of two or more components which is not split into pieces by any separating 2-sphere. In contrast, a link is called split if there is an embedded 2-sphere which separates some component(s) from others. We say a true link is $\epsilon$-embedded  if no two distinct components approach within distance $\epsilon$. The disjoint union of true links is the link type obtained by placing each  true link within its own topologically separating 2-sphere.

\begin{theorem}
	Suppose $L$ is the disjoint union $L = \sqcup_{i=1}^N L_i$, where each $L_i$ is a true link, and that $L$ embeds in the unit cube $I^3$ so that the embedding restricted to each $L_i$ is an $\epsilon$-imbedding, then there is a constant $\gamma > 0$ (independent of $L$) so that $N < e^{\gamma \epsilon^{-3}}$.
\end{theorem}

In particular, there is an upper bound to how many distinct true links can be placed in a box maintaining a fixed distance between the components of each of the true sublinks. I thank Fedya Manin for comments on an earlier draft.

\begin{proof}
	Give $I^3$ a fine triangulation of bounded geometry, i.e. $\text{max valence}:= v = O(1)$,\footnote{J.H.C.\ Whitehead \cite{whitehead40} introduced a method for constructing fine triangulations of bounded valence on any smooth manifold.} and whose number of vertices $V = O(\epsilon^{-3})$, so that all the dual cells all have diameter $< \frac{\epsilon}{2}$. We assume the link $L$ is in generic position w.r.t.\ this cell structure. For each $i$, $1 \leq i \leq N$, define a two-coloring $c_i$ of the $V$ dual cells using colors: black, and white. To form the coloring $c_i$, color every dual cell black iff it meets a component of $L_i$; color the remaining cells white. Define the codimension zero submanifold $B_i$, $1 \leq i \leq N$, to be the union of all black cells in the $c_i$ coloring. $B_i$ is a neighborhood of $L_i$ with $\pi_0(L_i) \ra \pi_0(B_i)$ an isomorphism, but not, in general, a tubular neighborhood.

	The number of possible colorings of this type is no more than $2^{O(\epsilon^{-3})}$, which we denote by $n$. This will be important since if the number $N$ of true links making up $L$ exceeds this quantity, then by the pigeonhole principle, for some $i \neq j$, $c_i$ and $c_j$ will be identical. In fact, if we assume for a contradiction that $N \geq e^{\gamma\epsilon^{-3}}$, then for large enough $\gamma$ we can ensure that there are at least $x > \mathrm{const.\,} V$ identical colorings $c_{i_1}, \dots ,c_{i_x}$ in our list of $N$ 2-colorings, for any positive constant const. We suppose this is the case for a sufficiently large const., and call their common black region B. The $x$ distinct true links, which, by re-ordering, we take to be $\{L_1,\dots,L_x\}$, whose neighborhood is $B$ must have the same number of components, call it $p$, which is also the number of components of $B$. Now consider how the splitting spheres for $L$, collectively called $S$, pass through this submanifold $B$ of $I^3$.

	We assume $S$ to be transverse to $Y := \de B$. We now use standard 3D techniques to modify $B$ to $B^\pr$ retaining $\{L_1,\dots,L_x\} \subset B^\pr$ and achieving $B^\pr \cap S = \varnothing$.

	Let us look first at a scc $\alpha$ in $S \cap Y$ which is innermost in $S$. Call $F$ an innermost disk of $S$ that $\alpha$ bounds. There are four cases: $\alpha$ may be trivial or nontrivial in $Y$ and $\operatorname{int}(F)$ may lie in $B$ or its complement $W$. Suppose $\alpha$ is trivial in $Y$ as witnessed by a disk $E \subset Y$ with boundary also $\alpha$. This $E \cup F$ bounds a 3-ball $C$ which is either in $B$, $C \subset B$, or oppositely $\operatorname{int}(C) \subset W$. In the first case, we modify $B$ by an ambient isotopy which subtracts $C$ from $B$, in the latter case we modify $B$ by adding $C$ to $B$; in both cases, call the result $B^\pr$. In the first case, no component of any of the links $\{L_1,\dots,L_x\}$ can lie in $C$, nor meet $\de C$, by the ``true link'' assumption and the isomorphism (actually, just the injectivity) on $\pi_0$. If one component of some $L_i$, $1 \leq i \leq x$, lay in $C$, then all of $L_i$ must, contradicting injectiity on $\pi_0$. So, an ambient isotopy supported near $B^\pr$ carries $B^\pr$ onto $B$ and fixes $\{L_1,\dots,L_x\}$. In the second case, since all $L_i$, $1 \leq i \leq x$ have neighborhoods equal to $B$, they must be disjoint from $C$. We abuse notation slightly by referring to all future modifications of $B$ as $B^\pr$.

	Now consider the case $\alpha$ is non-trivial in $Y$. If $F \subset B$, compress $Y$ along $F$ to delete an essential 2-handle from $B$. If, on the other hand, $\operatorname{int}(F) \subset W$, then add a 2-handle with core $F$ to $B$. In either case, the result $Y^\pr := \de B^\pr$ has had its complexity reduced. To measure this complexity, define a norm on closed oriented surfaces similar to the Gromov-Thurston norm as $\lVert Y \rVert = \sum (\abs{\chi(Y_k)} + 1)$, where the sum is taken over all components $Y_k$ of $Y$ of positive genus (exclude 2-spheres), and $\chi$ is Euler characteristic. Evidently compression (or surgery) on essential scc strictly reduces this norm. Now proceed to remove $S$-innermost scc of intersection with $Y$ either by isotopy when the scc are trivial in $Y$ or by compression/surgery when the scc are non-trivial. This process retains the injectivity $\pi_0(L_i) \ra \pi_0(B^\pr)$, $1 \leq i \leq x$, but not surjectivity. It is easy to give an upper bound $O(v \cd V)$ on the complexity of the initial surface $Y$, since the components of $B$ are obtained collectively by gluing up at most $V$ 3-cells along at most $v$ faces per gluing. The valence $v$ can be taken (using Whithead triangulations) to be a constant independent of $\epsilon$, so we may simply write the upper bound as $O(V)$. The handle addition/removal steps change the topology of $B$ and reduce the complexity, $\lVert Y^\pr \rVert < \lVert Y \rVert$, so there can only be $O(V)$ such steps. Some of these steps (the ones involving a compression of a homologically trivial $\alpha$) increase the number of components of B by one. Since originally $B$ has at most $O(V)$ components and only $O(V)$ steps add a component, then by the time we finish all steps, i.e.\ have modified $B$ to $B^\pr$ with $B^\pr \cap S = \varnothing$, $B^\pr$ will still only have $O(V)$ components.

	But, crucially, during each step, injectivity of $\pi_0(L_i) \ra \pi_0(B^\pr)$, $1 \leq i\leq x$, has been preserved. But this yields a contradiction if $\gamma$ is large enough. For distinct $i$ and $i^\pr$, $1 \leq i \leq i^\pr \leq x$, the two sublinks  $L_i$ and $L_{i^\pr}$ must have components in the same component of $B^\pr$, and these link components can be joined by an arc in $B^\pr$ which will not encounter $S$, contradicting the assumption that the two sublinks are separated by spheres of $S$.
\end{proof}

\section{Embedding Knots}

The same coloring method answers an analogous but simpler question about packing knots.

\begin{theorem}
	There is a constant $\delta > 0$ so that if a smooth ($C^2$) link $L$ in $I^3$ is a disjoint union of $N$ nontrivial knots $K_1, \dots ,K_N$, and each of these knots has an embedded normal bundle of radius $\epsilon > 0$. Then $N < e^{\delta\epsilon^{-3}}$.
\end{theorem}

\begin{proof}
	As before form the neighborhoods $B_1, \dots, B_N$ of the knots constituting $L$ by taking the union of appropriate cells (all of diameter $< \frac{\epsilon}{2}$) that the knots meet (transversely). For this application, it suffices to choose the cells to be that of a standard cubulation of $I^3$ shifted slightly to ensure it is in general position with $L$. Each $B_i$ is topologically a tubular neighborhood of its corresponding knot. For N larger than the stated estimate, the previous coloring argument shows that for some $i$ and $j$, $1 \leq i < j \leq N$, $B_i = B_j$. Uniqueness of tubular neighborhoods implies the $i$th and $j$th knots are of the same topological type, and that, as a two component link, either knot must be a framed normal push-off of the other. This is incompatible with the disjoint union (i.e.\ split) property of the link $L$ unless the two knots are actually unknots---which we have assumed not to be the case.
\end{proof}

\bibliography{references}

\begin{bibdiv}
\begin{biblist}

\bib{freedman24}{article}{
      author={Freedman, Michael~H.},
       title={Packing meets topology},
        date={2024},
     journal={Journal of Topology and Analysis},
      eprint={arXiv:2301.00295},
        note={To appear},
}

\bib{mp24}{article}{
      author={Manin, Fedor},
      author={Portnoy, Elia},
       title={On {F}reedman's link packings},
        date={2024},
     journal={Journal of Topology and Analysis},
      eprint={arXiv:2308.08064},
        note={To appear},
}

\bib{whitehead40}{article}{
      author={Whitehead, J.H.C.},
       title={On {$\mathbb{C}^1$}-complexes},
        date={1940},
     journal={The Annals of Mathematics},
      volume={41},
      number={2},
       pages={809\ndash 824},
}

\end{biblist}
\end{bibdiv}

\end{document}